\documentclass[letterpaper, 10 pt, conference]{ieeeconf}  

\IEEEoverridecommandlockouts                              
\usepackage{epsfig} 
\usepackage{amsmath,amssymb,amsfonts}
\usepackage{color}
\usepackage{balance}
\usepackage{booktabs,tabularx}
\usepackage{multirow}
\usepackage{mathtools}
\usepackage{pdfrender}
\usepackage{trfsigns}

\usepackage{centernot}
\usepackage{color}
\usepackage[acronym]{glossaries}
\usepackage[ruled]{algorithm2e}
\usepackage{graphicx,subcaption}
\usepackage{tabu}
\usepackage[official]{eurosym}

\graphicspath{{figures/}}

\newcommand{\R}{\mathbb{R}}

\newcommand{\N}{\mathbb{N}}

\newcommand{\mc}[1]{\mathcal{#1}}

\newcommand{\prob}{\mathbb{P}}

\newcommand{\bs}{\boldsymbol}

\newcommand{\col}{\mathrm{col}}

\newtheorem{theorem}{Theorem}

\newtheorem{definition}{Definition}
\newtheorem{proposition}{Proposition}
\newtheorem{lemma}{Lemma}

\newtheorem{remark}{Remark}
\newtheorem{assumption}{Assumption}
\newtheorem{standing}{Standing Assumption}

\makeglossaries
\newacronym{NEP}{NEP}{Nash equilibrium problem}
\newacronym{GNEP}{GNEP}{generalized Nash equilibrium problem}
\newacronym{iid}{i.i.d.}{independent and identically distributed}
\newacronym{w.r.t.}{w.r.t.}{with respect to}
\newacronym{a.s.}{a.s.}{almost surely}

\newglossaryentry{v-GNE}
{
	name={v-GNE},
	description={variational generalized Nash equilibrium},
	first={\glsentrydesc{v-GNE} (\glsentrytext{v-GNE})},
	plural={v-GNE},
	descriptionplural={variational generalized Nash equilibria},
	firstplural={\glsentrydescplural{v-GNE} (\glsentryplural{v-GNE})}
}

\newglossaryentry{GNE}
{
	name={GNE},
	description={generalized Nash equilibrium},
	first={\glsentrydesc{GNE} (\glsentrytext{GNE})},
	plural={GNE},
	descriptionplural={generalized Nash equilibria},
	firstplural={\glsentrydescplural{GNE} (\glsentryplural{GNE})}
}

\title{\LARGE \bf
On the robustness of equilibria in generalized aggregative games
}

\author{Filippo Fabiani, Kostas Margellos and Paul J. Goulart  
\thanks{The authors are with the Department of Engineering Science, University of Oxford, OX1 3PJ, United Kingdom {\tt \footnotesize (\{filippo.fabiani, kostas.margellos, paul.goulart\}@eng.ox.ac.uk)}. This work was partially supported through the Government’s modern industrial strategy by Innovate UK, part of UK Research and Innovation, under Project LEO (Ref. 104781).}%
}

\begin{document}

\maketitle
\thispagestyle{empty}
\pagestyle{empty}

\begin{abstract}

We address the problem of assessing the robustness of the equilibria in uncertain, multi-agent games. Specifically, we focus on generalized Nash equilibrium problems in aggregative form subject to linear coupling constraints affected by uncertainty with a possibly unknown probability distribution. Within a data-driven context, we apply the scenario approach paradigm to provide a-posteriori feasibility certificates for the entire set of generalized Nash equilibria of the game. Then, we show that assessing the violation probability of such set merely requires to enumerate the constraints that ``shape'' it. For the class of aggregative games, this results in solving a feasibility problem on each active facet of the feasibility region, for which we propose a semi-decentralized algorithm. We demonstrate our theoretical results by means of an academic example.

\end{abstract}

\section{Introduction}
Multi-agent decision making is an established paradigm to model and solve problems involving multiple heterogeneous (possibly selfish) agents with individual goals. Moreover, as part of the same population, such entities interact and potentially share common resources or compete for them, giving rise to a noncooperative setup.
In this context, equilibrium solution concepts based on game theory and, in particular, \glspl{GNEP} \cite{facchinei2007generalized}, provide a framework that encompasses many control engineering problems, e.g., communication and networks \cite{scutari2014real,facchinei2017feasible}, automated driving and traffic control \cite{smith1979existence,8672171}, smart grids and demand-side management \cite{ma2011decentralized,chen2014autonomous,9030152}. 

However, Nash equilibria are typically formalized in games with complete information, i.e., where the main ingredients (agents' cost functions and strategies, local and coupling constraints) are fully deterministic. Apparently, this might suggest a conceptual shift when dealing with real-world applications, since the latter are strongly affected by the presence of uncertainty, and therefore traditional equilibrium notions may no longer be appropriate. This motivates to seek for robust \gls{GNEP} reformulations, suitably accompanied by tailored equilibrium solution definitions.

By the pioneering work in \cite{harsanyi1962bargaining}, the literature on robust game theory divides into two main directions that depend on the available information (or working assumptions) around the uncertain parameter. Specifically, several results deal with uncertainty characterized by specific models of either their probability distribution \cite{couchman2005gaming,singh2016existence}, or the geometry of its support set \cite{aghassi2006robust,hayashi2005robust,perchet2020finding}. Conversely, there has been a recent development of data-driven (or distribution-free) robust approaches, see, e.g., \cite{9028952,paccagnan2019scenario,pantazis2020aposteriori}.

Within this data-driven context, the main results of the aforementioned papers characterize the robustness of equilibria to unseen realizations of the uncertain parameter by leveraging on the scenario approach paradigm \cite{campi2018introduction}. Originally conceived to provide a-priori feasibility guarantees associated with the optimal solution to an uncertain convex optimization problem \cite{calafiore2006scenario}, the scenario theory has been recently extended by means of an a-posteriori assessment of the feasibility risk to nonconvex decision-making problems \cite{campi2018general}. In a nutshell, the scenario theory establishes that the robustness of the solution to a given uncertain decision-making problem shall be assessed by solving an approximated, yet computationally tractable, problem that is built upon a finite number of observed realizations of the uncertainty.  

We aim at bridging the multi-agent generalized game theory with the data-driven scenario paradigm, in order to compute \gls{GNE} with quantifiable robustness properties in a distribution-free fashion. Specifically, we focus on the broad class of \glspl{GNEP} in aggregative setting (\S II), where the cost function of each agent depends on the average behaviour of the whole population and the strategies are coupled by means of (affine) coupling constraints affected by uncertainty with a possibly unknown probability distribution. Here, we contextualize and apply the probabilistic results in \cite{campi2018general} to provide a-posteriori feasibility certificates to the entire set of \glspl{v-GNE}, a popular subset of \glspl{GNE} \cite{cavazzuti2002nash}. Compared with the literature on robust data-driven game theory, our contributions can be summarized as follows.
\begin{itemize}
	\item Along the direction of \cite{pantazis2020aposteriori}, we focus on the entire set of equilibria, implicitly relaxing the assumption on the uniqueness of the equilibrium postulated in \cite{paccagnan2019scenario};
	
	\item We extend the results in \cite{9028952,pantazis2020aposteriori} providing a-posteriori robustness certificates for the set of \gls{GNE} rather than for the feasible set or Nash equilibria without coupling constraints. We also show that the resulting bounds are less conservative (\S III);
	
	\item The obtained probabilistic guarantees rely on the notion of support subsample, a key concept of the scenario approach theory. To compute these support subsamples we show that it is merely required to enumerate the constraints that ``shape'' the set of \gls{GNE}. An explicit representation of the unknown set of equilibria is therefore not needed (\S III);
	
	\item For the considered class of \glspl{GNEP}, we propose a structure-preserving, semi-decentralized algorithm to compute the number of minimal irreducible support subsamples \gls{w.r.t.} the set of \gls{GNE} (\S IV).
\end{itemize}

Finally, we validate the proposed theoretical results on an illustrative example (\S V).

\smallskip
\subsubsection*{Notation} 
$\N$ and $\R$ denote the set of natural and real numbers, respectively. For vectors $v_1,\dots,v_N\in\mathbb{R}^n$ and $\mc I=\{1,\dots,N \}$, we denote $\boldsymbol{v} \coloneqq (v_1 ^\top,\dots ,v_N^\top )^\top = \mathrm{col}((v_i)_{i\in\mc I})$ and $ \bs{v}_{-i} \coloneqq ( v_1^\top,\dots,v_{i-1}^\top,v_{i+1}^\top,\dots,v_{N}^\top )^\top =\col(( v_j )_{j\in\mc I\setminus \{i\}})$. With a slight abuse of notation, we also use $\bs{v} = (v_i,\bs{v}_{-i})$.  Given a matrix $A \in \R^{m \times n}$, $A^\top$ denotes its transpose, while for $A \in \R^{n \times n}$, $A \succ 0$ ($\succcurlyeq 0$) implies that $A$ is symmetric and positive (semi)-definite. For a given set $\mc{S} \subseteq \R^n$, $\textrm{bdry}(\mc{S})$ denotes its boundary.
If $\mc{S}$ is closed and convex, the normal cone of $\mc{S}$ evaluated at some $\bs{x}$ is the set-valued mapping $\mc{N}_{\mc{S}} : \R^n \to 2^{\R^n}$, defined as $\mc{N}_{\mc{S}}(\bs{x}) \coloneqq \{	d \in \R^n \mid d^\top (\bs{y} - \bs{x}) \leq 0, \; \forall \bs{y} \in \mc{S}	\}$ if $\bs{x} \in \mc{S}$, $\mc{N}_{\mc{S}}(\bs{x}) \coloneqq \emptyset$ otherwise. A mapping $F:\R^n \rightarrow \R^n$ is monotone if $(F(\bs{x}) - F(\bs{y}))^\top(\bs{x} - \bs{y}) \geq 0$ for all $\bs{x}, \bs{y} \in \R^n$. $\mc{C}^1$ is the class of continuously differentiable functions.


\section{Mathematical setup and problem statement}
We start by formalizing the data-driven, uncertain game considered. Then, we mathematically define the problem addressed, and finally recall some key results for the class of \gls{v-GNE} characterizing \glspl{GNEP} in aggregative form.

\subsection{Aggregative game formulation}
We consider a noncooperative, multi-agent game whose $N$ players are indexed by the set $\mc{I} \coloneqq \{1, \ldots,N\}$. Let $x_i \in \R^{n_i}$ be the decision vector of the $i$-th player, locally constrained to a set $\mc{X}_i \subseteq \R^{n_i}$. In this context, each player aims at minimizing a predefined cost function $J_i : \R^n \to \R$, $n \coloneqq \sum_{i \in \mc{I}} n_i$, while satisfying a set of coupling constraints among the agents affected by the realization of an uncertain vector $\delta$, encoded by the set $\mc{X}_{\delta} \subseteq \R^n$. Specifically, $\delta$ takes values in the set $\Delta \subseteq \R^\ell$, endowed with a $\sigma$-algebra $\mc{D}$ and distributed according to $\prob$, a possibly unknown probability measure over $\mc{D}$. This results in the following family of mutually coupled optimization problems:

$$
\forall i \in \mc{I} : \left\{
\begin{aligned}
&\underset{x_i \in \mc{X}_i}{\textrm{min}} & & J_i(x_i,  \bs{x}_{-i})\\
&\hspace{.1cm}\textrm{ s.t. } & & (x_i, \bs{x}_{-i}) \in \mc{X}_{\delta}, \, \delta \in \Delta.
\end{aligned}	
\right.
$$

For computational purposes, hereinafter we consider each cost function to be in aggregative form and quadratic, while $\mc{X}_{\delta}$ is a polyhedral set for every realization of $\delta$, i.e.,
$$
\begin{aligned}
	J_i &\coloneqq  \tfrac{1}{2} x^\top_i Q_i x_i + (\tfrac{1}{N}\textstyle\sum_{j \in \mc{I} \setminus \{i\}} C_{i,j} x_j + q_i)^\top  x_i, \; \forall i \in \mc{I},\\
	\mc{X}_{\delta} &\coloneqq \{\bs{x} \in \R^n \mid A(\delta) \, \bs{x} \leq b(\delta)\}, \; \forall \delta \in \Delta,
\end{aligned}
$$
where $Q_i \succ 0$, $C_{i,j} \in \R^{n_i \times n_j}$ for all $(i,j) \in \mc{I}^2$, $q_i \in \R^{n_i}$, while $A : \Delta \to \R^{m \times n}$ and $b : \Delta \to \R^{m}$. In view of the considered structure, it follows immediately that every $J_i(\cdot, \bs{x}_{-i})$ is a convex function of class $\mc{C}^1$, for any $\bs{x}_{-i} \in \R^{(n - n_i)}$, for all $i \in \mc{I}$. Then, given the linear structure of $\mc{X}_{\delta}$, we note that it can be equivalently defined by the set of inequalities $A_i(\delta) x_i + \textstyle\sum_{j \in \mc{I} \setminus \{i\}} A_j(\delta) x_j \leq b(\delta)$, with $A_i : \Delta \to \R^{m \times n_i}$, for all $i \in \mc{I}$ and for all $\delta \in \Delta$. For the remainder, we postulate the following assumption.

\smallskip
\begin{standing}
	For all $i \in \mc{I}$, $\mc{X}_i \subseteq \R^{n_i}$ is a polytopic set.
	\hfill$\square$
\end{standing}
\smallskip

To conclude, we note that the polytopic set encompassing all deterministic, local constraints $\mc{X} \coloneqq \prod_{i \in \mc{I}} \mc{X}_i$, can also be rewritten in compact form as $\mc{X} \coloneqq \{\bs{x} \in \R^n \mid H \bs{x} \leq h\}$, for some $H$ and $h$ obtained by concatenating the matrices and vectors that define the local constraint sets, $\mc{X}_i$.

\subsection{Scenario-based \gls{GNEP}}
The noncooperative game considered directly falls within the set of jointly convex \glspl{GNEP} \cite[Def.~2]{facchinei2007generalized}, and we consider a data driven approach to asses the robustness of a set of equilibria to such game. Specifically, let $\delta_K \coloneqq \{\delta^{(k)}\}_{k \in \mc{K}} = \{\delta^{(1)}, \ldots, \delta^{(K)}\} \in \Delta^K$ be a finite collection of $K \in \N \cup \{0\}$ \gls{iid} samples of $\delta$, $\mc{K} \coloneqq \{1,2,\ldots,K\}$, hereinafter referred to as $K$-multisample. The scenario-based \gls{GNEP} $\Gamma$ is defined as the tuple $\Gamma \coloneqq (\mc{I}, (\mc{X}_i)_{i \in \mc{I}}, (J_i)_{i \in \mc{I}}, \delta_K)$, encoded by the following family of optimization problems:
\begin{equation}\label{eq:single_prob_aggregative}
\forall i \! \in \! \mc{I} \! : \! \left\{
\begin{aligned}
&\underset{x_i \in \mc{X}_i}{\textrm{min}} & & \tfrac{1}{2} x^\top_i Q_i x_i + (\tfrac{1}{N} \textstyle\sum_{j \in \mc{I} \setminus \{i\}} C_{i,j} x_j + q_i)^\top x_i \\
&\hspace{.1cm}\textrm{ s.t. } & & A(\delta^{(k)}) \, \bs{x} \leq b(\delta^{(k)}), \, \text{ for all } k \in \mc{K}.
\end{aligned}	
\right.
\end{equation}

For any $\delta^{(k)} \in \delta_K$, define the set $\mc{X}_{\delta^{(k)}} \coloneqq \{\bs{x} \in \R^n \mid A(\delta^{(k)}) \, \bs{x} \leq b(\delta^{(k)}) \}$, while $\mc{X}^K_{i}(\bs{x}_{-i}) \coloneqq \{x_i \in \mc{X}_i \mid  (x_i, \bs{x}_{-i}) \in \cap_{k \in \mc{K}} \mc{X}_{\delta^{(k)}} \} $ and $\mc{X}_K \coloneqq \cap_{k \in \mc{K}} \mc{X}_{\delta^{(k)}}  \cap \mc{X}$. We consider the following notion of equilibrium for $\Gamma$.

\smallskip
\begin{definition}\label{def:GNE}
	Let $\delta_K \in \Delta^K$ be any $K$-multisample. The collective vector of strategies $\bs{x}^\ast \in \mc{X}_{K}$ is a \gls{GNE} of $\Gamma$ in \eqref{eq:single_prob_aggregative} if, for all $i \in \mc{I}$,
	$$
	J_i (x^\ast_i,  \bs{x}^\ast_{-i}) \leq \underset{y_i \in \mc{X}^K_{i}(\bs{x}^\ast_{-i})}{\mathrm{min}} \, J_i (y_i,  \bs{x}^\ast_{-i}).
	$$ 
	\hfill$\square$
\end{definition}
\smallskip

Clearly, given the dependence on the set of $K$ realizations $\delta_K \in \delta_K$, any equilibrium of $\Gamma$ is a random variable itself.

Now, let $\Omega_\delta$ be the set of equilibria induced by $\delta \in \Delta$. In the spirit of \cite[Def.~4]{pantazis2020aposteriori}, we investigate the violation probability of the set of equilibria of a scenario-based \gls{GNEP}, according to the definition given next.

\smallskip
\begin{definition}\label{def:violation_set}
	The violation probability of a set of \gls{GNE}, $\Omega$, is defined as
	\begin{equation}\label{eq:violation_set}
	V(\Omega) \coloneqq \prob\{\delta \in \Delta \mid \Omega \not\subseteq \Omega_\delta \}.
	\end{equation}
	\hfill$\square$
\end{definition} 
\smallskip

Specifically, the random variable $V(\Omega)$ encodes the robustness of the set $\Omega$ to the uncertain parameter $\delta$, i.e., given any reliability parameter $\epsilon \in (0,1)$, we say that $\Omega$ is $\epsilon$-robust if $V(\Omega) \leq \epsilon$. Here, the condition $\Omega \not\subseteq \Omega_\delta$ means that, once $\delta$ is drawn, at least one element in $\Omega$ is not an equilibrium any more. Thus, along the lines of \cite{campi2018general}, by relying on the observations of the uncertain parameter, i.e., the $K$-multisample $\delta_K$, our goal is to evaluate the violation probability of the set of equilibria $\Omega_K$. For the remainder, we restrict the set $\Omega_K$ to correspond to the set of \gls{v-GNE} of the scenario-based \gls{GNEP} \eqref{eq:single_prob_aggregative}, as described in the next section.

\subsection{Characterization of \gls{v-GNE}}
A popular subset of \gls{GNE} of a given game $\Gamma$ is the one of \gls{v-GNE}, characterized as the set of equilibria providing ``larger social stability'' \cite[\S 5]{cavazzuti2002nash}. Specifically, the set of \gls{v-GNE} corresponds to the set of collective strategies that solve the variational inequality associated with the scenario-based \gls{GNEP} in \eqref{eq:single_prob_aggregative}. Thus, given the $K$-multisample $\delta_K$, the set of \gls{v-GNE} coincides with the solution set to VI$(\mc{X}_K, F)$, where $\mc{X}_K$ is the feasible set and $F:\R^n \to \R^n$ is the so-called game mapping, constructed by stacking the partial derivatives of $J_i$, i.e., $F(\bs{x}) \coloneqq \col((\nabla_{x_i} J_i(x_i, \bs{x}_{-i}))_{i \in \mc{I}})$, given by
$$\Omega_{K} \coloneqq \{\bs{x} \in \mc{X}_K \mid (\bs{y} - \bs{x})^\top F(\bs{x}) \geq 0, \; \forall \bs{y} \in \mc{X}_K \}.$$

In our aggregative setting with quadratic cost functions, the game mapping turns out to be affine in the collective vector of strategies $\bs{x}$, i.e., $F(\bs{x}) = M \bs{x} + q$, where $M \in \R^{n \times n}$ and $q \in \R^n$ are defined as:
$$
M \! \coloneqq \! \left[\begin{array}{cccc}
Q_1 & \tfrac{1}{N}C_{1,2} & \cdots & \tfrac{1}{N}C_{1,N}\\
\tfrac{1}{N}C_{2,1} & Q_2 & \cdots & \tfrac{1}{N}C_{2,N}\\
\vdots & \vdots & \ddots &\vdots\\
\tfrac{1}{N}C_{N,1} & \tfrac{1}{N}C_{N,2} & \cdots & Q_N
\end{array}\right]\!, \, q  \coloneqq \! \left[\begin{array}{c}
q_1\\
q_2\\
\vdots\\
q_N
\end{array}
\right].
$$

\smallskip
\begin{standing}\label{ass:monotone_game}
	The mapping $F:\R^n \to \R^n$ is monotone.
	\hfill$\square$
\end{standing}
\smallskip

We remark that an affine mapping is monotone if and only if  $(M + M^\top) \succcurlyeq 0$. This can be guaranteed by, e.g., assuming equivalent bilateral interactions among agents, $C_{i,j} = C_{j,i}$, for all $(i,j) \in \mc{I}^2$ (in addition to $Q_i \succ 0$, for all $i \in \mc{I}$).

Now, we recall some results available in the literature on affine variational inequalities, which will be key in the remainder of the paper. 
Specifically, let us consider first the game $\Gamma$ in the absence of the coupling constraints, and let us focus on the (deterministic) \gls{NEP} associated to \eqref{eq:single_prob_aggregative} with $\mc{X}_K = \mc{X}$, which reads as
\begin{equation}\label{eq:NEP_single_prob_aggregative}
\forall i \! \in \! \mc{I} \!:\!
\underset{x_i \in \mc{X}_i}{\textrm{min}} \; \tfrac{1}{2} x^\top_i Q_i x_i + (\tfrac{1}{N} \textstyle\sum_{j \in \mc{I} \setminus \{i\}} C_{i,j} x_j + q_i)^\top x_i. 
\end{equation}

The set of variational Nash equilibria to such \gls{NEP}, namely $\Omega_{0} \coloneqq \Omega_{\delta^{(0)}}$,  coincides with the set of solutions to a linearly constrained, affine variational inequality problem, and hence is characterized by the following lemma that combines \cite[Lemma~2.4.14,  Th.~2.4.15]{facchinei2007finite}, \cite[Lemma~1, Th.~2]{gowda1994boundedness}.

\smallskip
\begin{lemma}\label{lemma:agg_GNEP}
	Let $M \succcurlyeq 0$. Then, the following statements hold true:
	
	\begin{enumerate}
		\item[(i)] $\Omega_{0}$ is a bounded polyhedral set;
		\item[(ii)] There exist a vector $c \in \R^n$ and a constant $d \geq 0$ such that, for all $\bs{x} \in \Omega_{0}$, $(M + M^\top) \bs{x} = c$ and $\bs{x}^\top M \bs{x} = d$;
		
		\item[(iii)] Let $\omega(\bs{x}) \coloneqq \textrm{min}_{\bs{y} \in \mc{X}} \; \bs{y}^\top (M \bs{x} + q)$, and let $\mc{P} \coloneqq \{ \bs{x} \in \mc{X} \mid \omega(\bs{x}) - (d + q^\top \bs{x}) \geq 0 \}$. Then
		$$
		\Omega_{0} \coloneqq \{ \bs{x} \in \mc{P} \mid (M + M^\top) \bs{x} = c \}.
		$$
	\end{enumerate}
	\hfill$\square$
\end{lemma}
\smallskip

By noticing that $\mc{P}$ is a polyhedral set, roughly speaking the set of Nash equilibria $\Omega_{0}$ contains the feasible strategies that span $(M + M^\top)$, and it is characterized by the two invariants $c$ and $d$. We note that, given any $\delta_K \in \Delta^K$, Lemma~\ref{lemma:agg_GNEP}(iii) allows to introduce coupling constraints, and characterize the set of \gls{v-GNE}, $\Omega_{K}$. Specifically, we have
\begin{equation}\label{eq:generic_set_equilibria}
	\Omega_{K} \coloneqq \{ \bs{x} \in \R^n \mid (M + M^\top) \bs{x} = c \} \cap \mc{P}_K,
\end{equation}
where $\mc{P}_K \coloneqq \{ \bs{x} \in \mc{X}_K \mid \omega(\bs{x}) - (d + q^\top \bs{x}) \geq 0 \}$, and the function $\omega(\bs{\cdot})$ is restricted to the feasible set $\mc{X}_K \subseteq \mc{X}$, which accounts for the coupling constraints. Finally, we recall that $\bs{x}^\ast \in \Omega_{K} \iff -F(\bs{x}^\ast) = -(M \bs{x}^\ast + q) \in \mc{N}_{\mc{X}_K}(\bs{x}^\ast)$.

\smallskip
\begin{remark}
	In view of Standing Assumption~\ref{ass:monotone_game}, we have $M \succcurlyeq 0$. When $M \succ 0$, the mapping $F$ is strictly monotone and hence the scenario-based \gls{GNEP} admits a unique equilibrium that, in general, can not be characterized as in Lemma~\ref{lemma:agg_GNEP}. The results showed next focus on the general case, i.e., $F$ monotone mapping with $M \succcurlyeq 0$, while the other case follows straightforwardly. In fact, if $M \succ 0$, Lemma~\ref{lemma:ifonlyif} and Theorem~\ref{th:VI} below still hold, by requiring only Assumption~\ref{ass:nonemptiness} to be imposed, thus relaxing Assumption~\ref{ass:nondeg}.
	\hfill$\square$
\end{remark}

\section{Probabilistic feasibility for a set of \gls{GNE}}
In this section, we first recall some key concepts and results of the scenario approach theory, and then discuss how to extend them to a set-oriented framework. Successively, we provide bounds on the violation probability related to the set of equilibria of the scenario-based \gls{GNEP} $\Gamma$ in \eqref{eq:single_prob_aggregative}.

\subsection{A weak connection among sets of \gls{GNE}}
Recent developments in the scenario approach literature have led to a-posteriori probabilistic feasibility guarantees for abstract decision problems \cite{campi2018general} (see Theorem 1 therein), which is based on the two following conditions:
\begin{enumerate}
	\item[(i)] For all $K \in \N \cup \{0\}$ and all $\delta_K \in \Delta^K$, the decision of an abstract problem is unique;
	\item[(ii)] The decision taken while observing $K$ realizations shall be \emph{consistent} for all the collected situations $k \in \mc{K}$ \cite[Assumption~1]{campi2018general}.
\end{enumerate}

Specifically, \cite[Th.~1]{campi2018general} studies the distribution of $V(\theta^\ast_K)$, where $\theta^\ast_K$ is the unique solution to the abstract decision problem computed after observing $K$ realizations of the uncertain parameter, and finds a suitable (probabilistic) bound $1-\beta$ guaranteeing that $V(\theta^\star_K) \leq \epsilon$ holds, for some $\beta \in (0,1)$.

Since the randomized \gls{GNEP} in \eqref{eq:single_prob_aggregative} is a decision problem, a key step to apply the probabilistic feasibility bound in \cite[Th.~1]{campi2018general} to the entire set of \gls{GNE} is to extend the conditions above to embrace the scenario-based generalized aggregative game $\Gamma$. To this end, in view of Definition~\ref{def:violation_set}, we mimic the steps made in \cite{campi2018general} by focusing on set-oriented decisions.

In the scenario-based \gls{GNEP} considered, our decision is a set and, specifically, we let correspond to the set of equilibria, $\Omega_K$. Then, in view of item \textrm{(i)}, guaranteeing the uniqueness of the set of equilibria for $\Gamma$ in \eqref{eq:single_prob_aggregative} is implicit since, for any $K$-multisample $\delta_K \in \Delta^K$, there is naturally a single set of equilibria $\Omega_K$, which is a nonempty, compact and convex set. This follows immediately from \cite[Th.~2.3.5]{facchinei2007finite}, as $\mc{X}_K$  is a bounded polyhedral set and $F$ is a continuous, monotone mapping. Therefore, let us consider a single-valued mapping $\Theta_K : \Delta^K \to 2^\mc{X}$ that, given a specific set of realizations $\delta_K$, returns the set of equilibria to the scenario-based \gls{GNEP} in \eqref{eq:single_prob_aggregative}, i.e., $\Omega_{K} \coloneqq \Theta_K(\delta^{(1)}, \ldots, \delta^{(K)}) = \Theta_K(\delta_K)$. When $K = 0$, $\Theta_0$ has no argument, and it is to be understood that it returns the set of equilibria of the deterministic \gls{NEP} in \eqref{eq:NEP_single_prob_aggregative}. In view of item \textrm{(ii)} above, we envision the following set-oriented counterpart of \cite[Ass.~1]{campi2018general}.

\smallskip
	``For all $K \in \N$ and for all $\delta_K \in \Delta^K$, $\Theta_K(\delta_K) \subseteq \Omega_{\delta^{(k)}}$, for all $k \in \mc{K} \cup \{0\}$.''
\smallskip

In the proposed analogy, we let the admissible decision for the situation represented by $\delta$ to coincide with the set of equilibria $\Omega_\delta$, which is clearly a subset of the feasible set $\mathcal{X}_\delta$ shaped by the uncertain parameter. 
Next, we show that the above set-oriented counterpart of \cite[Ass.~1]{campi2018general} holds true for the scenario-based \gls{GNEP} in \eqref{eq:single_prob_aggregative}. Given the specific structure of the problem addressed, in view of Lemma~\ref{lemma:agg_GNEP}, we postulate the following assumptions on the set of equilibria.

\begin{figure}[t]
	\centering
	\includegraphics[width=\columnwidth]{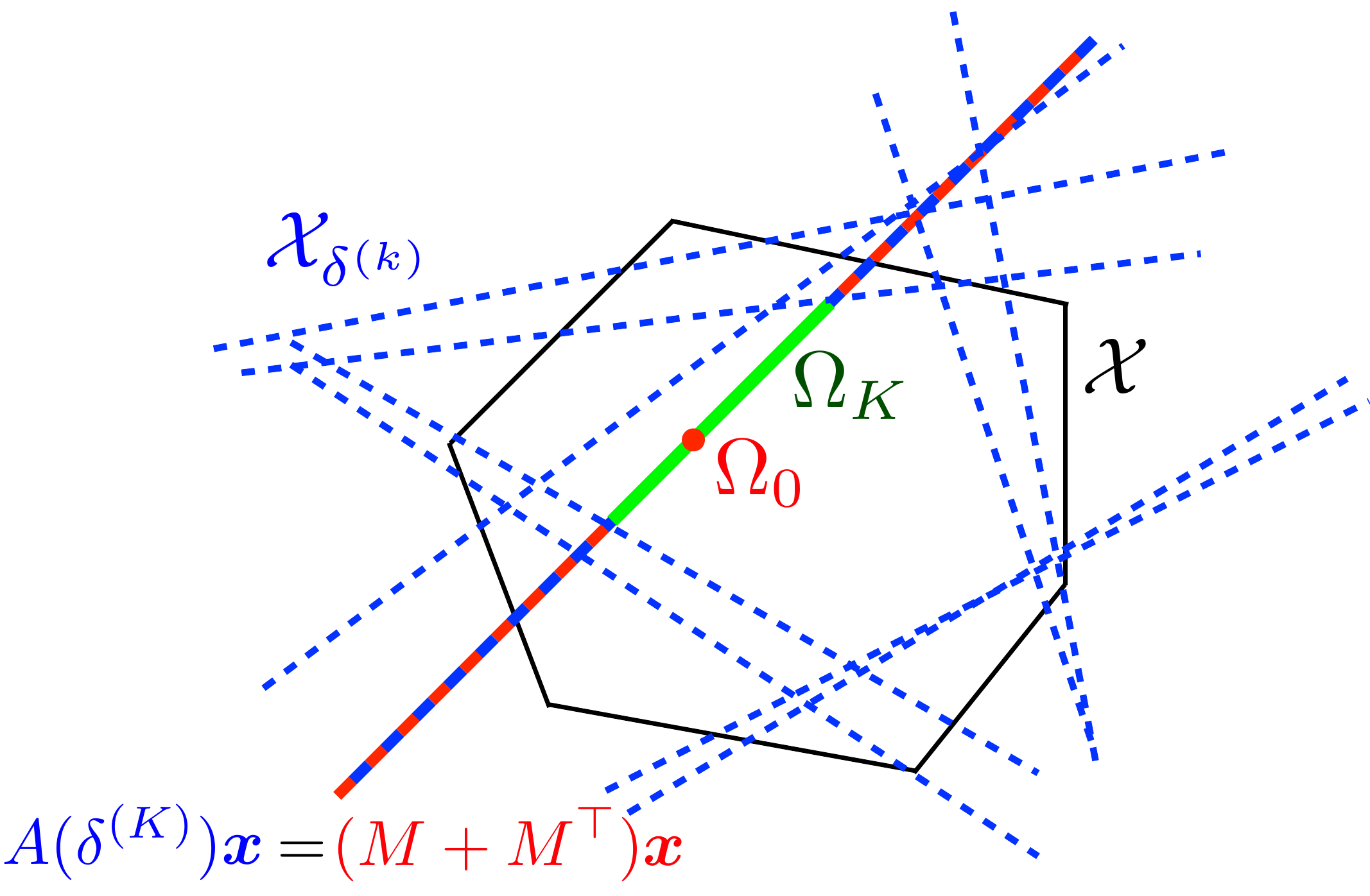}
	\caption{Schematic two-dimensional representation of the type of degenerate cases that, due to Assumption 2, can only happen with probability zero. In this case, $\Omega_{0}, \ldots, \Omega_{K-1}$ are singletons, however, if the $K$-th sample overlaps with the affine set in \eqref{eq:generic_set_equilibria} (assuming  $b(\delta^{(K)}) = c = \bs{0}$), it might generate additional equilibria belonging to $\textrm{bdry}(\mc{X}_{K})$, as formalized in the proof of Lemma~\ref{lemma:ifonlyif}. Thus, $\Omega_{K}$ is no longer a singleton and may lie entirely on the $\textrm{bdry}(\mc{X}_K)$.}
	\label{fig:deg_scenario_GNEP}
\end{figure}

\smallskip
\begin{assumption}\label{ass:nonemptiness}
	For all $K \in \N \cup \{0\}$, $\Omega_{K} \cap \mc{X}_{\delta}$ is nonempty, for any $\delta \in \Delta$.
	\hfill$\square$
\end{assumption}
\smallskip
\begin{assumption}\label{ass:nondeg}
	For all $\bs{x} \in \R^n$, $\prob\{\delta \in \Delta \mid A(\delta) \bs{x} - b(\delta) =  (M + M^\top) \bs{x} - c\} = 0$.
	\hfill$\square$
\end{assumption}
\smallskip

Nonemptiness of $\Omega_{K}$ is reasonable as we aim at quantifying robustness to unseen scenarios, while Assumption~\ref{ass:nondeg}  is a non-degeneracy condition often imposed in the scenario approach literature \cite[Ass.~6]{garatti2019risk}. It rules out, indeed, the possibility that a new affine coupling constraint corresponding to $\delta$ overlaps with the equilibria subspace $(M+M^\top)\bs{x} - c$, allowing such situations to occur with probability zero (see Fig. 1 for a graphical representation). This requirement is satisfied for all probability distributions $\prob$ that admit a density function.
Pictorially, generating samples gives rise to shared constraints that ``shape'' the set of equilibria, as represented in Fig.~\ref{fig:scenario_GNEP}. With this in mind, we are now in the position to prove the main result that links the set of \gls{GNE} of \eqref{eq:single_prob_aggregative} across the samples scenarios, thus establishing (probabilistically) the set-oriented counterpart of \cite[Ass.~1]{campi2018general}. 

\begin{figure}[t]
	\centering
	\includegraphics[width=\columnwidth]{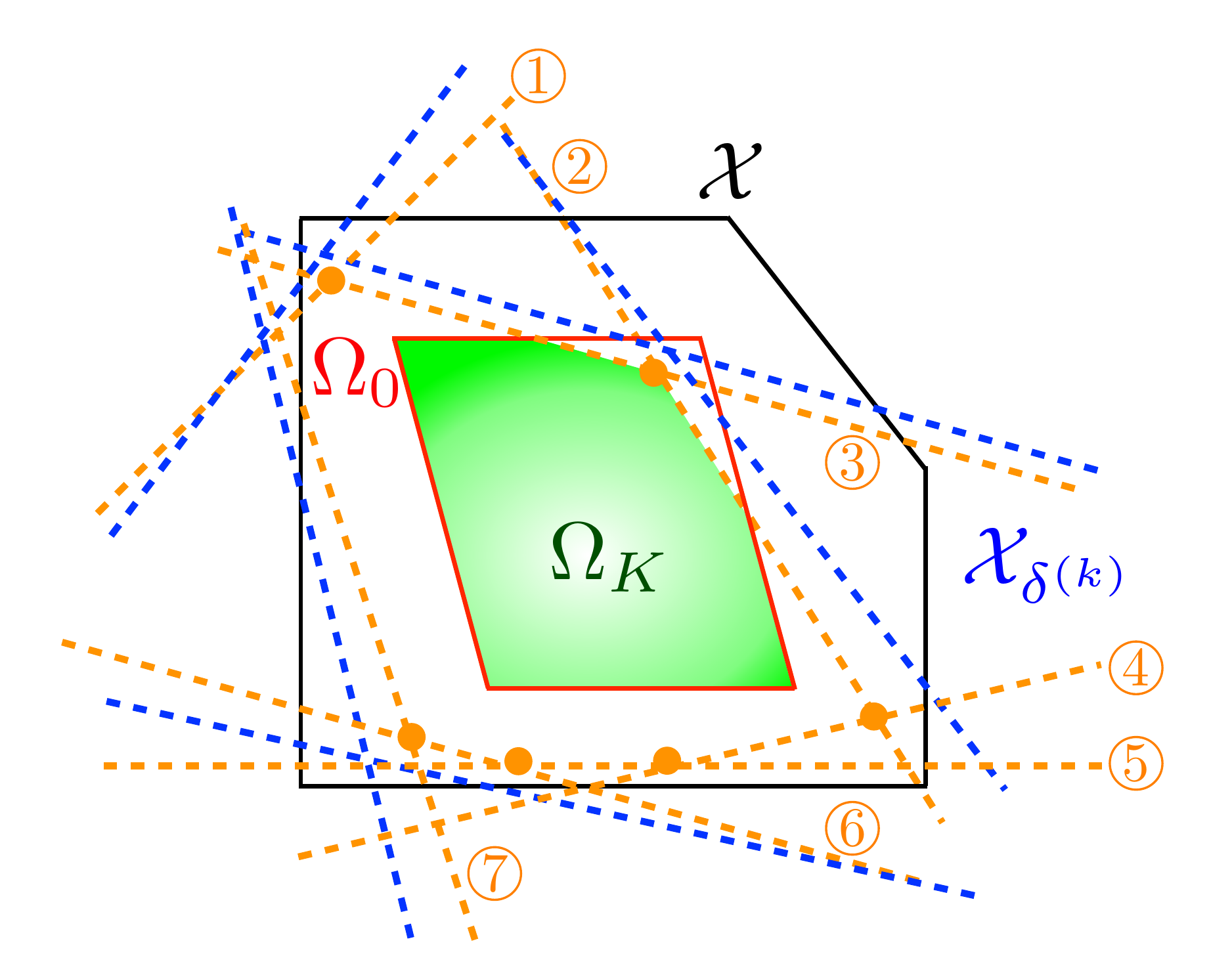}
	\caption{The set of \gls{GNE}, $\Omega_{K}$, can be ``shaped'' by the set of linear constraints, $\mc{X}_{\delta^{(k)}}$, $k \in \mc{K}$. 
		Specifically, by referring to Definition~\ref{def:support_sub}, the labelled dashed orange lines define the support subsample for $\delta_K$ \gls{w.r.t.} to $\mc{X}_K$, which in this case correspond to the the active samples that shape the feasibility region. Their intersections in $\mc{X}$ are denoted by orange dots.}
	\label{fig:scenario_GNEP}
\end{figure}

\smallskip
\begin{lemma}\label{lemma:ifonlyif}
		Let Assumption~\ref{ass:nonemptiness} and \ref{ass:nondeg} hold true. Then, for all $K \in \N$ and for all $\delta_K \in \Delta^K$, $\Theta_K(\delta_K) \subseteq \Omega_{\delta^{(k)}}$ \gls{a.s.}, for all $k \in \mc{K} \cup \{0\}$.
	\hfill$\square$
\end{lemma}
\begin{proof}
		Given any $K \in \N$ and any associated $K$-multisample $\delta_K \in \Delta^K$, let $\bar{k} $ be an arbitrary index belonging to $\mc{K}$. The mapping $\Theta_{\bar{k}}(\delta^{(1)}, \ldots, \delta^{(\bar{k})})$ returns the set of equilibria $\Omega_{\bar{k}}$, while once drawn the ($\bar{k}+1$)-th sample, we have $\Omega_{\bar{k}+1} \coloneqq \Theta_{\bar{k}+1}(\delta^{(1)}, \ldots, \delta^{(\bar{k})}, \delta^{(\bar{k}+1)})$. Note that, in view of Assumption~\ref{ass:nonemptiness}, both sets are guaranteed to be nonempty and are of the form defined in \eqref{eq:generic_set_equilibria}, i.e., generated by the intersection between an affine and a bounded polyhedral set. We show now that $\Omega_{\bar{k}+1} \subseteq \Omega_{\bar{k}}$, and then the statement will follow by induction over $\bar{k} \in \mc{K}$ by noticing further that $\Theta_0 \eqqcolon \Omega_{0} \supseteq \Omega_{1} \supseteq \ldots \supseteq \Omega_{K} \eqqcolon \Theta_K(\delta_K)$.
		
		On one hand, any $\bs{x}^\ast$ that is a \gls{GNE} for $\Gamma$ on $\mc{X}_{\bar{k}}$ and such that $\bs{x}^\ast \in \mc{X}_{\delta^{(\bar{k} +1)}}$, also belongs to $\Omega_{\bar{k}+1}$. To see this, recall the definition of $\Omega_{\bar{k}+1}$ in \eqref{eq:generic_set_equilibria}: the inclusion is clearly true for the affine part, $(M + M^\top) \bs{x}^\ast = c$, while if $\bs{x}^\ast \in \mc{P}_{\bar{k}}$ and $\bs{x}^\ast \in \mc{X}_{\delta^{(\bar{k} +1)}}$, then $\bs{x}^\ast \in \mc{P}_{\bar{k} + 1}$ in view of the structure of $\omega(\cdot)$, along with the convexity and compactness of each set involved. Now, let $\mc{X}_{\bar{k} + 1} \coloneqq \mc{X}_{\bar{k}} \cap \mc{X}_{\delta^{(\bar{k}+1)}}$. In view of the properties of the normal cone, if there exists some \gls{GNE} $\bs{x}^\ast$ such that $\bs{x}^\ast \in \Omega_{\bar{k} +1}$, but $\bs{x}^\ast \notin \Omega_{\bar{k}}$, it must happen that $\bs{x}^\ast \in \textrm{bdry}(\Omega_{\bar{k} +1})$. In fact, $-F(\bs{x}^\ast) \in \mc{N}_{\mc{X}_{\bar{k} + 1}}$ and $-F(\bs{x}^\ast) \notin \mc{N}_{\mc{X}_{\bar{k}}}$ if and only if $F(\bs{x}^\ast) \neq \{\bs{0}\}$, and this is possible at the boundary of $\mc{X}_{\bar{k} + 1}$ only, which in view of the compactness and convexity of each set corresponds to the boundary of $\Omega_{\bar{k} +1}$ (see also Fig.~\ref{fig:deg_scenario_GNEP}). Thus, $\Omega_{\bar{k} +1}$ can be represented as the union of two sets. Specifically, the first set gathers all those points that were equilibria for the game with $\bar{k}$-samples and remain feasible for the constraint corresponding to $\bar{k}+1$, while the second one contains all those points that did not belong to $\Omega_{\bar{k}}$ and may lie on the boundary of $\Omega_{\bar{k} +1}$, i.e.,
		\begin{equation}\label{eq:sets}
			\begin{aligned}
			\Omega_{\bar{k} +1} &= \{\bs{x} \in \mc{X}_{\bar{k} + 1} \mid \bs{x} \in \Omega_{\bar{k}}\} \cup \left\{\bs{x} \in \textrm{bdry}(\mc{X}_{\bar{k} +1}) \mid \right. \\
			&\left.\bs{x} \notin \Omega_{\bar{k}}, (\bs{y} - \bs{x})^\top F(\bs{x}) \geq 0 \; \forall \bs{y} \in \mc{X}_{\bar{k} +1}\right\}.
			\end{aligned}
		\end{equation}
	
		In view of Assumption~\ref{ass:nonemptiness}, $\{\bs{x} \in \mc{X}_{\bar{k} + 1} \mid \bs{x} \in \Omega_{\bar{k}}\} \subseteq \Omega_{\bar{k}}$ is nonempty, for any $\mc{X}_{\delta^{(\bar{k}+1)}}$ and $\bar{k} \in \mc{K}$. Finally, it follows that the second set in \eqref{eq:sets} is empty \gls{a.s.}, since it contains all points that are on the boundary of $\mc{X}_{\bar{k}+1}$ and belong to the affine set $(M + M^\top) \bs{x} = c$ which is of measure zero due to Assumption~\ref{ass:nondeg}, thus concluding the proof.
\end{proof}
\smallskip

We finally remark again that Fig.~\ref{fig:deg_scenario_GNEP} shows an example where the second set in \eqref{eq:sets} would not be of measure zero.

\subsection{A-posteriori probabilistic feasibility guarantees for $\Omega_{K}$}
The following definition is at the core of scenario approach theory and crucial for our subsequent developments.

\smallskip
\begin{definition}\textup{\cite[Def.~2]{campi2018general}}\label{def:support_sub}
	Given a $K$-multisample $\delta_K \in \Delta^K$, a support subsample $S \subseteq \delta_K$ is a $p$-tuple of elements extracted from $\delta_K$, i.e., $S \coloneqq \{\delta^{(k_1)}, \ldots, \delta^{(k_p)}\}$, $k_1 < \ldots < k_p$, which gives the equilibria of the original sample, i.e.,
	$$
	\Theta_p(\delta^{(k_1)}, \ldots, \delta^{(k_p)}) = \Theta_K(\delta^{(1)}, \ldots, \delta^{(K)}).
	$$
	\hfill$\square$
\end{definition}
\smallskip

Moreover, a support subsample $S$ is said to be irreducible if no further elements can be removed from $S$ without leaving the solution unchanged. With a slight abuse of notation, in the remainder we will refer to the notion of support subsample for $\delta_K \in \Delta^K$ \gls{w.r.t.} either $\mc{X}_K$, or $\Omega_{K}$.

In general, an algorithm that determines a support subsample can be defined as $\Upsilon_K : \delta_K \to \{k_1, \ldots, k_p\}$, $k_1 < \ldots < k_p$, such that $\{\delta^{(k_1)}, \ldots, \delta^{(k_p)}\}$ is a support subsample for $\delta_K$. Let us denote with $s_K \coloneqq |\Upsilon_K(\delta_K)|$ its cardinality. Note that $s_k$ is a random variable itself as it depends on $\delta_K$.

Thus, given any $K$-multisample $\delta_K \in \Delta^K$, the following result provides an a posteriori bound of the violation probability in \eqref{eq:violation_set} for the entire set of equilibria, $\Omega_K$.

\smallskip
\begin{theorem}\label{th:VI}
	Let Assumption~\ref{ass:nonemptiness} and \ref{ass:nondeg} hold true, fix some $\beta \in (0,1)$. Let $\varepsilon : \mc{K} \cup \{0\} \to [0, 1]$ be a function such that
	$$
	\left\{\begin{aligned}
	& \varepsilon(K) = 1,\\
	& \sum_{h = 0}^{K - 1} \left( \begin{array}{c}
	K\\
	h
	\end{array} \right) (1 - \varepsilon(h))^{K - h} = \beta.
	\end{aligned}
	\right.
	$$
	Then, for any $\Theta_K$, $\Upsilon_K$ and probability $\prob$, it holds that 
	\begin{equation}\label{eq:prob_feas_boud}
		\prob^K \{\delta_K \in \Delta^K \mid V(\Omega_{K}) > \varepsilon(s_K) \} \leq \beta.
	\end{equation}
	\hfill$\square$
\end{theorem}
\begin{proof}
	By leveraging on Lemma~\ref{lemma:ifonlyif}, the proof follows as a corollary of \cite[Th.~1]{campi2018general}.
\end{proof}

\smallskip
\begin{remark}
	As evident from \eqref{eq:prob_feas_boud}, to asses the robustness of the set of equilibria $\Omega_{K}$, one does not need to dispose of a full characterization of $\Omega_{K}$, namely an algorithm $\Theta_K(\cdot)$, but rather the number of support subsamples $s_K$, computed by means of $\Upsilon(\cdot)$. In the next section, we provide a possible algorithm $\Upsilon(\cdot)$ for the scenario-based \gls{GNEP} in \eqref{eq:single_prob_aggregative}.
	\hfill$\square$
\end{remark}
\smallskip

The following result provides an upper bound for $V(\Omega_{K})$.

\smallskip
\begin{proposition}\label{prop:better_performance}
	Given any $K \in \N \cup \{0\}$ and $\delta_K \in \Delta^K$, let $s_K$ and $v_K$ be the number of (possibly irreducible) support subsample for $\delta_K$, evaluated \gls{w.r.t.} $\Omega_{K}$ and $\mc{X}_K$, respectively. Then, $s_K \leq v_K$, and therefore $V(\Omega_{K}) \leq V(\mc{X}_{K})$.
	\hfill$\square$
\end{proposition}
\begin{proof}
	Given the linearity of both local and coupling constraints defining the feasible set of the game $\Gamma$ in \eqref{eq:single_prob_aggregative}, it follows from Definition~\ref{def:support_sub} that some sample $\delta^{(k)}$ is of support for $\delta_K$ \gls{w.r.t.} $\mc{X}_K$ if $\mc{X}_{\delta^{(k)}}$ is active on $\textrm{bdry}(\mc{X}_K)$, i.e., $\textrm{bdry}(\mc{X}_{\delta^{(k)}}) \cap \textrm{bdry}(\mc{X}_K) \neq \emptyset$. On the other hand, $\delta^{(k)}$ is of support \gls{w.r.t.}  $\Omega_{K}$ if  $\textrm{bdry}(\mc{X}_{\delta^{(k)}}) \cap \Omega_{K} \neq \emptyset$ (see Fig.~\ref{fig:scenario_GNEP} for a graphic illustration). Since, in general, $\Omega_{K} \subseteq \mc{X}_K = \cap_{k \in \mc{K}} \mc{X}_{\delta^{(k)}} \cap \mc{X}$, those samples that are of support for $\delta_K$ \gls{w.r.t.} $\Omega_{K}$, are of support \gls{w.r.t.} $\mc{X}_K$, but not viceversa. Therefore, $s_K \leq v_K$. Finally, since $\varepsilon(\cdot)$ in Theorem~\ref{th:VI} is an increasing function, we have $V(\Omega_{K}) \leq V(\mc{X}_{K})$ as desired.
\end{proof}
\smallskip

\begin{remark}
	The result in Proposition~\ref{prop:better_performance} implies that, given the same $K$-multisample $\delta_K \in \Delta^K$, \eqref{eq:prob_feas_boud} provides tighter bounds compared to \cite[Cor.~7]{pantazis2020aposteriori}, since we focus on the set of equilibria rather than on the entire feasibility set.
	\hfill$\square$
\end{remark}

\section{Computational aspects}
Next, we propose a structure-preserving, semi-decentralized algorithm to compute the number of support subsample \gls{w.r.t.} $\Omega_{K}$. In view of Theorem~\ref{th:VI}, $s_K$ is a crucial quantity to assess the risk associated with the entire set $\Omega_{K}$.

\begin{algorithm}[!t]
	\caption{Computation of the number of support subsample for aggregative \glspl{GNEP}}\label{alg:support_agg_game}
	\DontPrintSemicolon
	\SetArgSty{}
	\SetKwFor{ForAll}{for all}{do}{end forall}
	\smallskip
	\textbf{Initialization:}
	\begin{itemize} \setlength{\itemindent}{0.8cm}
		\item[(\texttt{S0.1})] Set $s_K \coloneqq 0$, identify
		$$\mc{A}_K \coloneqq \{k \in \mc{K} \mid \textrm{bdry}(\mc{X}_{\delta^{(k)}}) \cap \textrm{bdry}(\mc{X}_K) \neq \emptyset \}$$
		\item[(\texttt{S0.2})] Run $\Phi(\delta_0)$ to compute $\bs{x}_0 \in \mc{X}$, set $d \coloneqq \bs{x}^\top_0 M \bs{x}_0$ and $c \coloneqq (M + M^\top) \bs{x}_0$
	\end{itemize}
	\smallskip
	\textbf{Iteration $(i \in \mc{A}_K)$:} \\
	\begin{itemize}\setlength{\itemindent}{.5cm}
		\item[(\texttt{S1})] Solve the feasibility problem:
		\begin{equation}\label{eq:feas_prob_GNEP}
		\left\{
		\begin{aligned}
		&\underset{(\lambda, \bs{x}) \in \R^{n+1}}{\textrm{min}} & & 0\\
		&\hspace{.45cm}\textrm{ s.t. } & & h^\top \lambda + q^\top \bs{x} + d \leq 0,\\
		&&& H^\top \lambda - M^\top \bs{x} + c + q = 0,\\
		&&& \lambda \geq 0, \bs{x} \in \mc{X}_{K} \cap \textrm{bdry}(\mc{X}_{\delta^{(i)}}).
		\end{aligned}	
		\right.
		\end{equation}
		
		\item[(\texttt{S2})] If $\exists (\lambda, \bs{x})$ that solves \eqref{eq:feas_prob_GNEP}, set $s_K \coloneqq s_K + 1$
	\end{itemize}
\end{algorithm}

Specifically, by leveraging on Lemma~\ref{lemma:agg_GNEP}, in the case of \gls{GNEP} in aggregative setting the computation of the (minimal) number of support subsample \gls{w.r.t.} $\Omega_{K}$ reduces to solving a feasibility problem on the augmented space $\R^{n + 1}$. 
An outline of a complete procedure can be found in Algorithm~\ref{alg:support_agg_game}, where, given any $K$-multisample $\delta_K \in \Delta^K$, $\Phi : \Delta^K \to 2^{\mc{X}_K}$ can be seen as any iterative algorithm available in the literature that allows to compute an equilibrium solution to the aggregative \gls{GNEP} in \eqref{eq:single_prob_aggregative}, e.g., \cite{salehisadaghiani2016distributed,belgioioso2017semi,liang2017distributed}. Specifically, while (\texttt{S0.1}) allows to identify the active facets of the convex polytope $\mc{X}_K$ \cite{ziegler2012lectures}, (\texttt{S0.2}) requires to solve the \gls{NEP} in \eqref{eq:NEP_single_prob_aggregative}, here identified by $\delta_0 = \emptyset$. In this way, computing an equilibrium of the \gls{NEP} allows us to define the quantities $d$ and $c$, which characterize every point in $\Omega_{0}$ (and therefore of $\Omega_{K}$), also shaping the feasibility set in \eqref{eq:feas_prob_GNEP}. Successively, (\texttt{S1}) requires to solve a feasibility problem on each active facet identified at (\texttt{S0.1}), where $\bs{x} \in \mc{X}_{K} \cap \textrm{bdry}(\mc{X}_{\delta^{(i)}})$ translates into an equality constraint in view of the affine constraints involved, while (\texttt{S2}) increments the counter $s_K$ in case the problem at (\texttt{S1}) is feasible.
We next state and prove the main result related with Algorithm~\ref{alg:support_agg_game}.

\smallskip
\begin{proposition}\label{prop:irreducible_set_agg_GNEP}
	Let Assumption~\ref{ass:nonemptiness} and \ref{ass:nondeg} hold true. For any $K \in \N$ and $\delta_K \in \Delta^K$, Algorithm~\ref{alg:support_agg_game} returns $s^\ast_K$, the cardinality of the minimal, irreducible support subsample $\delta_K$ \gls{w.r.t.} the entire set of equilibria, $\Omega_{K}$.
	\hfill$\square$
\end{proposition}
\begin{proof}
	First note that, in the setting of the scenario-based \gls{GNEP} in \eqref{eq:single_prob_aggregative}, $\mc{A}_K$ denotes the minimal, irreducible support subsample for $\delta_K$ \gls{w.r.t.} the convex polytope $\mc{X}_K$.  Then, by following the consideration adopted within the proof of Proposition~\ref{prop:better_performance}, i.e., every $\delta^{(i)}$, $i \in \mc{A}_K$, is of support also  \gls{w.r.t.} $\Omega_{K}$ if and only if $\textrm{bdry}(\mc{X}_{\delta^{(i)}}) \cap \Omega_{K} \neq \emptyset$. To check this condition for each $\delta^{(i)}$ it is sufficient to compute a solution (if one exists) on the active region of $\mc{X}_K$ associated with $\mc{X}_{\delta^{(i)}}$. Since, in general, $\Omega_{K} \subseteq \Omega_{0}$ (both bounded polyhedral sets), in view of Lemma~\ref{lemma:agg_GNEP} every equilibrium solution in $\Omega_{K}$ is characterized by: i) the invariance property with parameter $c$, which is computed, together with $d$, for the \gls{NEP}, i.e., \eqref{eq:single_prob_aggregative} with no coupling constraints; ii) shall lie into $\mc{P}_K$, defined in \eqref{eq:generic_set_equilibria}. Let us consider now the Lagrange dual optimization problem associated with $\omega(\bs{x})$ given by
	\begin{equation}\label{eq:dual_agg_GNEP}
	\left\{
	\begin{aligned}
	&\underset{\lambda \geq 0}{\textrm{max}} & & - h^\top \lambda\\
	&\hspace{0cm}\textrm{ s.t. } & & H^\top \lambda + M \bs{x} + q = 0.\\
	\end{aligned}	
	\right.
	\end{equation}
	
	In view of weak duality \cite{boyd2004convex}, $\bs{x} \in \mc{P}_K$ (recall the definition of $\mc{P}_K$ below \eqref{eq:generic_set_equilibria}) if there exists some $\lambda \geq 0$ such that \eqref{eq:dual_agg_GNEP} is feasible and $- h^\top \lambda - (d + q^\top \bs{x}) \geq 0$ is satisfied as $\omega(\bs{x}) \geq -h^\top \lambda$ for any such $\lambda$.
	Thus, by combining the equality in \eqref{eq:dual_agg_GNEP} and $(M+M^\top) \bs{x} = c$ to obtain the second constraint in \eqref{eq:feas_prob_GNEP}, computing an equilibrium on the boundary of an active constraint $\mc{X}_K \cap \textrm{bdry}(\mc{X}_{\delta^{(i)}})$ reduces to finding a feasible pair $(\lambda, \bs{x})$ for the convex optimization problem in \eqref{eq:feas_prob_GNEP}. Finally, $s_K$ increases only if such a feasibility problem has a solution, excluding all those samples $\mc{X}_{\delta^{(i)}}$ that does not intersect $\Omega_{K}$. The minimality follows as a consequence of the fact that $\mc{A}_K$ is the minimal support subsample for the polytope $\mc{X}_K$.
\end{proof}
\smallskip

\begin{remark}
	As tailored for \gls{GNEP} in aggregative form, Algorithm~\ref{alg:support_agg_game} requires to run the adopted iterative procedure $\Phi(\delta_K)$ once, and to solve \eqref{eq:feas_prob_GNEP} by means of some distributed algorithm $|\mc{A}_K|$-times, with $|\mc{A}_K| \leq K$. This clearly improves \gls{w.r.t.} the greedy algorithms proposed in \cite[\S II]{campi2018general} and \cite[\S III]{paccagnan2019scenario}, which would require running $\Phi(\delta_K)$ $K$-times.
	\hfill$\square$
\end{remark}

\section{Illustrative example}
We choose an academic example to illustrate the introduced theoretical results. Specifically, we consider a two-player \gls{GNEP} in aggregative form with scalar decision variables and quadratic structure, i.e., we consider $N=2$ agents, with cost functions $J_1(x_1, x_2) \coloneqq \tfrac{1}{2} x^2_1 + (1 - x_2) x_1$, $J_2(x_1, x_2) \coloneqq \tfrac{1}{2} x^2_2 - (1 + x_1) x_2$, and $\mc{X}_{i} \coloneqq \{x_i \in \R \mid |x_i| \leq 2\}$, $i = 1, 2$. Here, $\bs{x} \coloneqq \col(x_1,x_2)$, and
$$
M  \coloneqq \left[\begin{array}{cc}
\phantom{-}1 & -1\\
-1 & \phantom{-}1
\end{array}\right], \, q  \coloneqq \left[\begin{array}{c}
\phantom{-}1\\
-1
\end{array}
\right],
$$
which guarantee the monotonicity of the game mapping $F$ as $M+M^\top \succcurlyeq 0$.
Thus, it turns out that $\Omega_{0} \coloneqq \{\bs{x} \in \mc{X} \mid x_2 - x_1 - 1 = 0\}$, $\mc{X} \coloneqq \mc{X}_1 \times \mc{X}_2$, and since $M \succcurlyeq 0$, every $\bs{x}^\ast \in \Omega_{0}$ is characterized by invariants $c \coloneqq \col(-2,2)$ and $d \coloneqq 1$ as in Lemma~\ref{lemma:agg_GNEP}. We assume each set $\mc{X}_\delta$ be defined by a random halfspace of the form $\delta_1 x_1 + \delta_2 x_2 \leq \delta_3$. Moreover, we assume that $\delta \coloneqq \col(\delta_1,\delta_2,\delta_3)$ follows a uniform distribution with support $\Delta \coloneqq [-4, 4] \times [-4, 4] \times [4, 10]  \subseteq \R^3$, shaping the feasible set $\mc{X}_\delta \cap \mc{X}$.

\begin{figure}
	\centering
	\includegraphics[width=\columnwidth]{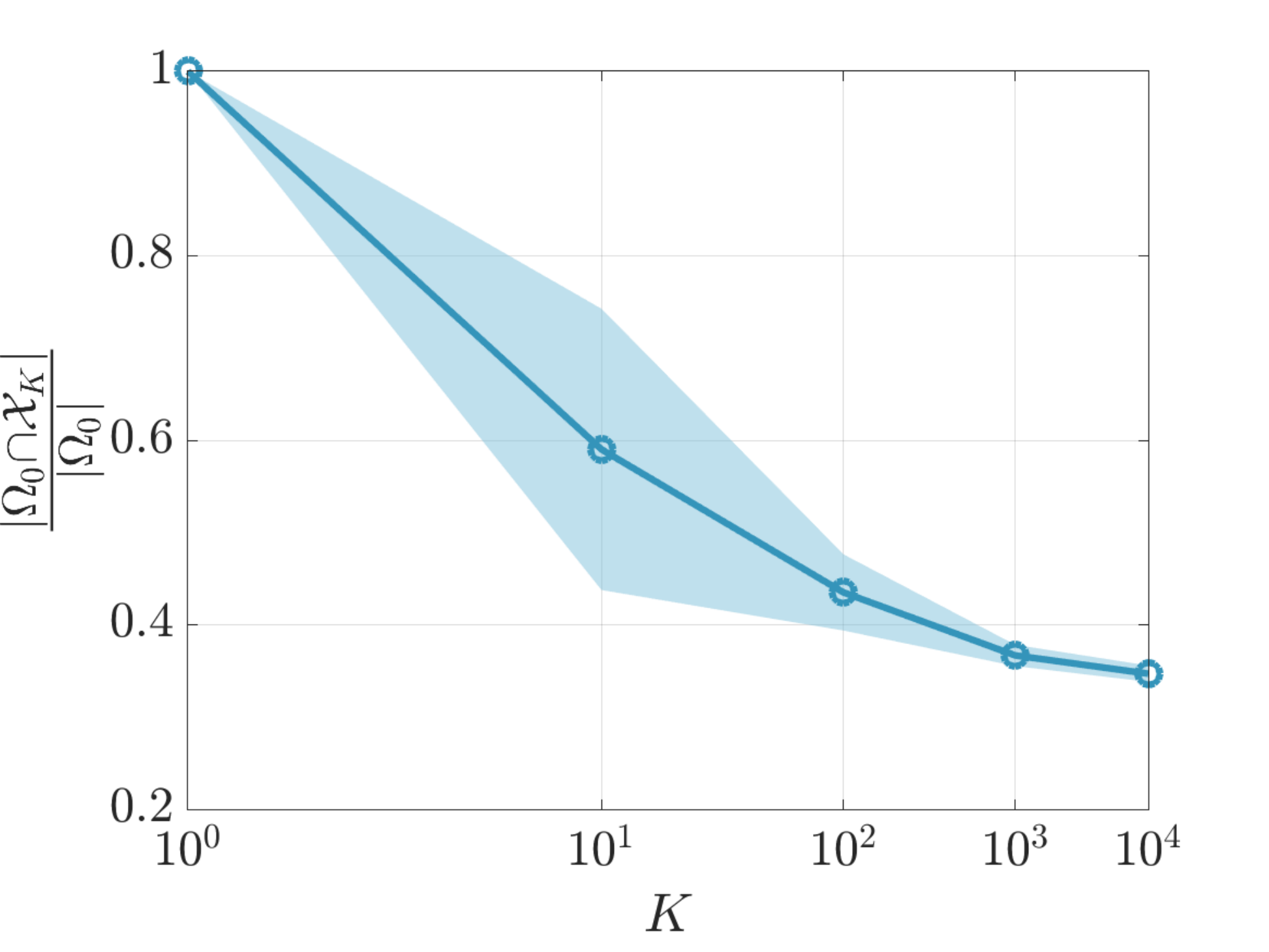}
	\caption{Size of $\Omega_K = \Omega_0 \cap \mathcal{X}_K$, normalized with the one of $\Omega_0$, as a function of the number of samples $K$. The solid line represents the average of $|\Omega_{0} \cap \mc{X}_K|/|\Omega_{0}|$ over $10$ numerical experiments, while the shaded area the standard deviation.}\label{fig:set_red}
\end{figure}

Then, given any $K$-multisample, the structure of $\Omega_{K}$ enable us to estimate $|\Omega_{K}|$ as the length of the interval $(M+M^\top) \bs{x} = c$ contained in $\mc{X}_K$, i.e., $|\Omega_k| = |\Omega_0 \cap \mathcal{X}_K|$. Thus, Fig.~\ref{fig:set_red} shows the average length of $\Omega_{K}$ over $10$ numerical experiments, normalized \gls{w.r.t.} the one of $\Omega_{0}$. Here, $\Omega_{K}$ shrinks as the number of samples grows, numerically supporting Lemma~\ref{lemma:ifonlyif}. Note that, in view of the structure of the support $\Delta$, as $K$ increases, the standard deviation of the uncertain parameter $\delta$ narrows around the average.

\begin{figure}
	\centering
	\includegraphics[width=\columnwidth]{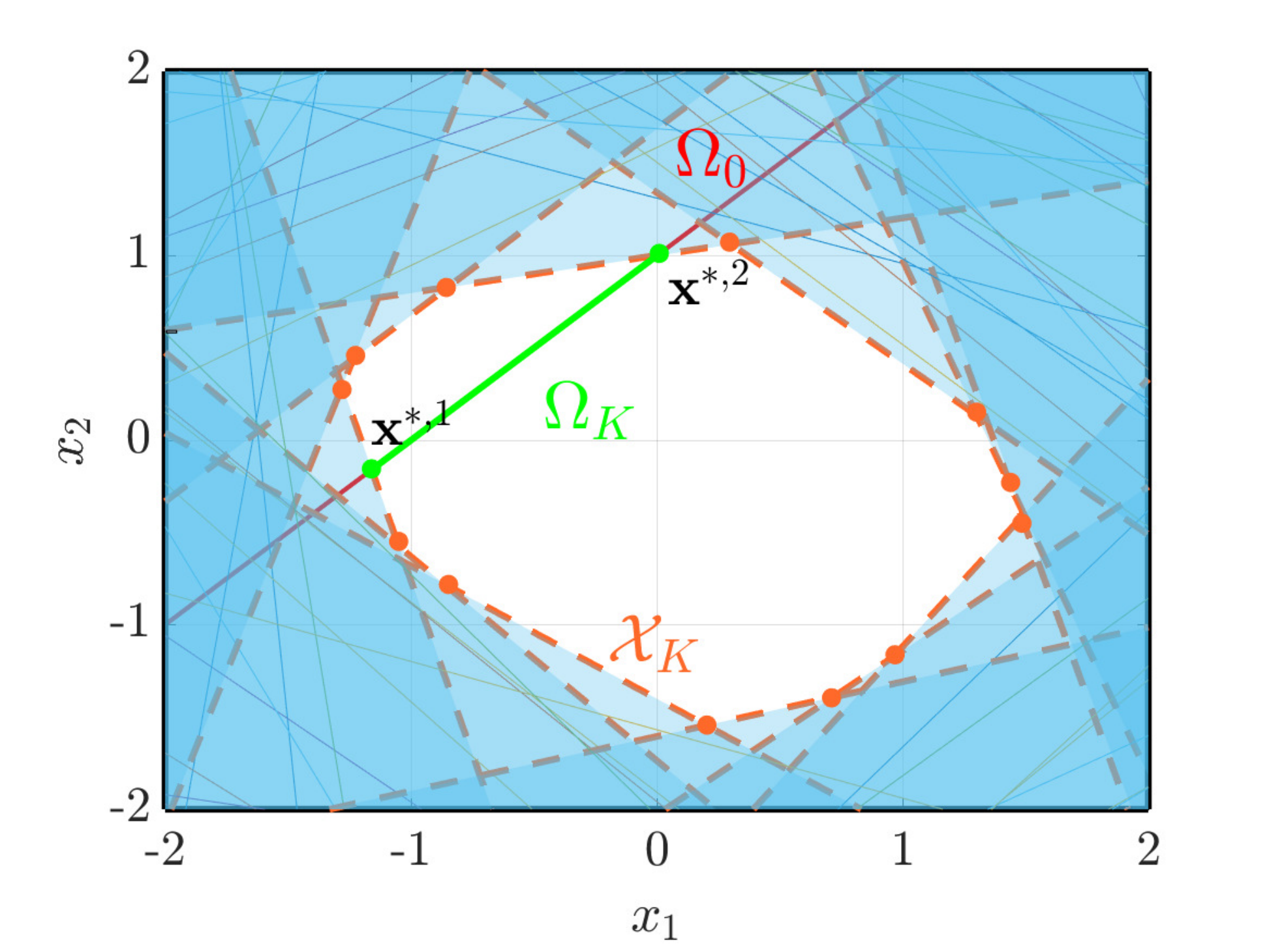}
	\caption{Sets obtained after drawing $100$ samples. The green dots $\bs{x}^{\ast,1}$ and $\bs{x}^{\ast,2}$ are the extrema of the set of \gls{GNE}, $\Omega_{100}$, while the orange dashed lines confine the feasible set, $\mathcal{X}_{100}$.}\label{fig:preliminary}
\end{figure}
\begin{figure}
	\centering
	\includegraphics[width=\columnwidth]{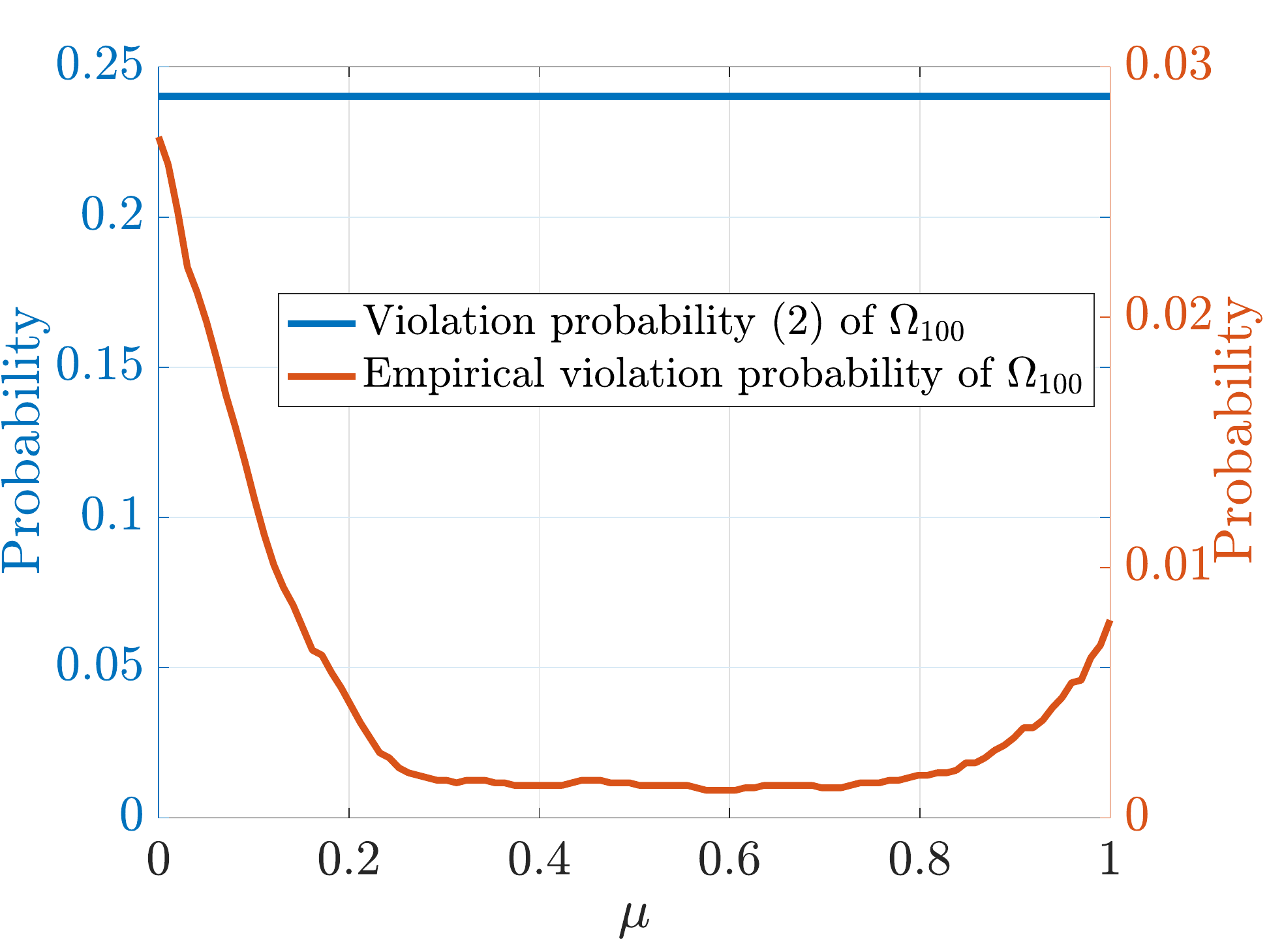}
	\caption{Comparison between theoretical and empirical violation probability for $\Omega_{100}$. After gridding $\Omega_{100}$ with granularity $0.01$, the empirical violation probability is evaluated for each grid-point against $10^4$ new samples.}\label{fig:theovsemp}
\end{figure}

We now compare the theoretical bounds provided in Theorem~\ref{th:VI}, by using $\beta = 10^{-6}$, with an empirical estimate of the violation probability in \eqref{eq:violation_set}. To this end, we generate $K = 100$ samples to obtain $\Omega_{100}$ in Fig.~\ref{fig:preliminary} and, after gridding the set of equilibria with granularity $0.01$, we compute the empirical violation of probability for each grid-point against $10^4$ new realizations of $\delta$. The theoretical violation level, encoded by the function $\varepsilon(\cdot)$ is analytically obtained by splitting $\beta$ evenly among the $100$ terms within the summation defined in Theorem~\ref{th:VI}. Given the structure of the problem, the family of equilibria in $\Omega_{100}$ corresponds to an interval, which can be parametrized by the points $(1 - \mu) \, \bs{x}^{\ast,1} + \mu \, \bs{x}^{\ast,2}$, for $\mu \in [0,1]$, where $\bs{x}^{\ast,1}$ and $\bs{x}^{\ast,2}$ are the extrema of $\Omega_{100}$ (see Fig.~\ref{fig:preliminary}).
As reported in Fig.~\ref{fig:theovsemp}, while the theoretical bound in \eqref{eq:prob_feas_boud}, determined by $s_{100} = 2$, provides an equivalent feasibility certificate for all the points in $\Omega_{100}$, the empirical violation probability is generally lower and attains the highest values close to $\bs{x}^{\ast,1}$  and $\bs{x}^{\ast,2}$. This is anticipated as closer to the boundary of the set higher probability of violation is expected.

\section{Conclusion and Outlook}
The scenario approach applied to robust game theory provides a numerically tractable framework to compute \gls{GNE} with quantifiable robustness properties in a distribution-free fashion. In the specific case of a \gls{GNEP} in aggregative form, we allow assessing the robustness properties of the entire set of generalized equilibria, thus relaxing the requirement for imposing a Nash equilibrium uniqueness assumption as typically performed in the literature. This merely requires to enumerate the active coupling constraints that intersect such set. Further extensions to other classes of \glspl{GNEP} and potential games, along with different algorithms to compute the number of support subsamples, a crucial quantity for the feasibility certificate, constitute topics of future work.

\balance
\bibliographystyle{IEEEtran}
\bibliography{20_CDC_scenario_GNEP}

\begin{thebibliography}{10}
\providecommand{\url}[1]{#1}
\csname url@samestyle\endcsname
\providecommand{\newblock}{\relax}
\providecommand{\bibinfo}[2]{#2}
\providecommand{\BIBentrySTDinterwordspacing}{\spaceskip=0pt\relax}
\providecommand{\BIBentryALTinterwordstretchfactor}{4}
\providecommand{\BIBentryALTinterwordspacing}{\spaceskip=\fontdimen2\font plus
\BIBentryALTinterwordstretchfactor\fontdimen3\font minus
  \fontdimen4\font\relax}
\providecommand{\BIBforeignlanguage}[2]{{%
\expandafter\ifx\csname l@#1\endcsname\relax
\typeout{** WARNING: IEEEtran.bst: No hyphenation pattern has been}%
\typeout{** loaded for the language `#1'. Using the pattern for}%
\typeout{** the default language instead.}%
\else
\language=\csname l@#1\endcsname
\fi
#2}}
\providecommand{\BIBdecl}{\relax}
\BIBdecl

\bibitem{facchinei2007generalized}
F.~Facchinei and C.~Kanzow, ``Generalized {N}ash equilibrium problems,''
  \emph{4OR}, vol.~5, no.~3, pp. 173--210, 2007.

\bibitem{scutari2014real}
G.~Scutari, F.~Facchinei, J.-S. Pang, and D.~P. Palomar, ``Real and complex
  monotone communication games,'' \emph{IEEE Transactions on Information
  Theory}, vol.~60, no.~7, pp. 4197--4231, 2014.

\bibitem{facchinei2017feasible}
F.~Facchinei, L.~Lampariello, and G.~Scutari, ``Feasible methods for nonconvex
  nonsmooth problems with applications in green communications,''
  \emph{Mathematical Programming}, vol. 164, no. 1-2, pp. 55--90, 2017.

\bibitem{smith1979existence}
M.~J. Smith, ``The existence, uniqueness and stability of traffic equilibria,''
  \emph{Transportation Research Part B: Methodological}, vol.~13, no.~4, pp.
  295--304, 1979.

\bibitem{8672171}
F.~{Fabiani} and S.~{Grammatico}, ``Multi-vehicle automated driving as a
  generalized mixed-integer potential game,'' \emph{IEEE Transactions on
  Intelligent Transportation Systems}, vol.~21, no.~3, pp. 1064--1073, March
  2020.

\bibitem{ma2011decentralized}
Z.~Ma, D.~S. Callaway, and I.~A. Hiskens, ``Decentralized charging control of
  large populations of plug-in electric vehicles,'' \emph{IEEE Transactions on
  Control Systems Technology}, vol.~21, no.~1, pp. 67--78, 2011.

\bibitem{chen2014autonomous}
H.~Chen, Y.~Li, R.~H. Louie, and B.~Vucetic, ``Autonomous demand side
  management based on energy consumption scheduling and instantaneous load
  billing: An aggregative game approach,'' \emph{IEEE Transactions on Smart
  Grid}, vol.~5, no.~4, pp. 1744--1754, 2014.

\bibitem{9030152}
C.~{Cenedese}, F.~{Fabiani}, M.~{Cucuzzella}, J.~M.~A. {Scherpen}, M.~{Cao},
  and S.~{Grammatico}, ``Charging plug-in electric vehicles as a mixed-integer
  aggregative game,'' in \emph{2019 IEEE 58th Conference on Decision and
  Control (CDC)}, 2019, pp. 4904--4909.

\bibitem{harsanyi1962bargaining}
J.~C. Harsanyi, ``Bargaining in ignorance of the opponent's utility function,''
  \emph{Journal of Conflict Resolution}, vol.~6, no.~1, pp. 29--38, 1962.

\bibitem{couchman2005gaming}
P.~Couchman, B.~Kouvaritakis, M.~Cannon, and F.~Prashad, ``Gaming strategy for
  electric power with random demand,'' \emph{IEEE Transactions on Power
  Systems}, vol.~20, no.~3, pp. 1283--1292, 2005.

\bibitem{singh2016existence}
V.~V. Singh, O.~Jouini, and A.~Lisser, ``Existence of {N}ash equilibrium for
  chance-constrained games,'' \emph{Operations Research Letters}, vol.~44,
  no.~5, pp. 640--644, 2016.

\bibitem{aghassi2006robust}
M.~Aghassi and D.~Bertsimas, ``Robust game theory,'' \emph{Mathematical
  Programming}, vol. 107, no. 1-2, pp. 231--273, 2006.

\bibitem{hayashi2005robust}
S.~Hayashi, N.~Yamashita, and M.~Fukushima, ``Robust {N}ash equilibria and
  second-order cone complementarity problems,'' \emph{Journal of Nonlinear and
  Convex Analysis}, vol.~6, no.~2, p. 283, 2005.

\bibitem{perchet2020finding}
V.~Perchet, ``Finding robust {N}ash equilibria,'' in \emph{Algorithmic Learning
  Theory}, 2020, pp. 725--751.

\bibitem{9028952}
F.~{Fele} and K.~{Margellos}, ``Probabilistic sensitivity of {N}ash equilibria
  in multi-agent games: a wait-and-judge approach,'' in \emph{2019 IEEE 58th
  Conference on Decision and Control (CDC)}, 2019, pp. 5026--5031.

\bibitem{paccagnan2019scenario}
D.~Paccagnan and M.~C. Campi, ``The scenario approach meets uncertain game
  theory and variational inequalities,'' in \emph{2019 IEEE 58th Conference on
  Decision and Control (CDC)}, 2019, pp. 6124--6129.

\bibitem{pantazis2020aposteriori}
G.~Pantazis, F.~Fele, and K.~Margellos, ``A posteriori probabilistic
  feasibility guarantees for {N}ash equilibria in uncertain multi-agent
  games,'' \emph{IFAC-PapersOnLine}, 2020, (To appear).

\bibitem{campi2018introduction}
M.~C. Campi and S.~Garatti, \emph{Introduction to the scenario approach}.\hskip
  1em plus 0.5em minus 0.4em\relax SIAM, 2018, vol.~26.

\bibitem{calafiore2006scenario}
G.~C. Calafiore and M.~C. Campi, ``The scenario approach to robust control
  design,'' \emph{IEEE Transactions on Automatic Control}, vol.~51, no.~5, pp.
  742--753, 2006.

\bibitem{campi2018general}
M.~C. Campi, S.~Garatti, and F.~A. Ramponi, ``A general scenario theory for
  nonconvex optimization and decision making,'' \emph{IEEE Transactions on
  Automatic Control}, vol.~63, no.~12, pp. 4067--4078, 2018.

\bibitem{cavazzuti2002nash}
E.~Cavazzuti, M.~Pappalardo, and M.~Passacantando, ``Nash equilibria,
  variational inequalities, and dynamical systems,'' \emph{Journal of
  Optimization Theory and Applications}, vol. 114, no.~3, pp. 491--506, 2002.

\bibitem{facchinei2007finite}
F.~Facchinei and J.~S. Pang, \emph{Finite-dimensional variational inequalities
  and complementarity problems}.\hskip 1em plus 0.5em minus 0.4em\relax
  Springer Science \& Business Media, 2007.

\bibitem{gowda1994boundedness}
M.~S. Gowda and J.~S. Pang, ``On the boundedness and stability of solutions to
  the affine variational inequality problem,'' \emph{SIAM Journal on Control
  and Optimization}, vol.~32, no.~2, pp. 421--441, 1994.

\bibitem{garatti2019risk}
S.~Garatti and M.~Campi, ``Risk and complexity in scenario optimization,''
  \emph{Mathematical Programming}, pp. 1--37, 2019.

\bibitem{salehisadaghiani2016distributed}
F.~Salehisadaghiani and L.~Pavel, ``Distributed {N}ash equilibrium seeking: A
  gossip-based algorithm,'' \emph{Automatica}, vol.~72, pp. 209--216, 2016.

\bibitem{belgioioso2017semi}
G.~Belgioioso and S.~Grammatico, ``Semi-decentralized {N}ash equilibrium
  seeking in aggregative games with separable coupling constraints and
  non-differentiable cost functions,'' \emph{IEEE Control Systems Letters},
  vol.~1, no.~2, pp. 400--405, 2017.

\bibitem{liang2017distributed}
S.~Liang, P.~Yi, and Y.~Hong, ``Distributed {N}ash equilibrium seeking for
  aggregative games with coupled constraints,'' \emph{Automatica}, vol.~85, pp.
  179--185, 2017.

\bibitem{ziegler2012lectures}
G.~M. Ziegler, \emph{Lectures on polytopes}.\hskip 1em plus 0.5em minus
  0.4em\relax Springer Science \& Business Media, 2012, vol. 152.

\bibitem{boyd2004convex}
S.~Boyd and L.~Vandenberghe, \emph{Convex optimization}.\hskip 1em plus 0.5em
  minus 0.4em\relax Cambridge university press, 2004.

\end{thebibliography}

\end{document}